\documentclass{siamart171218}
\usepackage{lipsum}
\usepackage{amsfonts}
\usepackage{amssymb}
\usepackage{graphicx}
\usepackage{dsfont}
\usepackage{subcaption} % for subfigures
\usepackage{epstopdf}
\usepackage{todonotes}
\usepackage{amsopn}
\usepackage{algorithmic}
\usepackage{verbatim}
\ifpdf
  \DeclareGraphicsExtensions{.eps,.pdf,.png,.jpg}
\else
  \DeclareGraphicsExtensions{.eps}
\fi

\newsiamremark{remark}{Remark}
\newsiamremark{notation}{Notation}
\newsiamremark{hypothesis}{Hypothesis}
\newsiamthm{claim}{Claim}
\newsiamthm{NFM}{Nonlinear Forward Model}
\newsiamthm{ip}{Inverse Problem}
\newsiamthm{assumption}{Assumption}
\headers{Uniqueness and reconstructions from 
  intensity correlations}{T. Hohage, M. Karimi, B. Müller}

\title{Holographic X-ray Phase Contrast Imaging with Partial Coherence: Uniqueness and Reconstructions from 
  Intensity Correlations
\thanks{
\funding{This work was supported by Deutsche Forschungsgemeinschaft through Project 432680300 (CRC 1456/C03)}}}

\author{Thorsten Hohage\thanks{Institute for Numerical and Applied Mathematics, Lotzestra{\ss}e 16-18, 37083 University of Göttingen, Germany and Max-Planck Institute for Solar System Research, 37077 Göttingen, Germany (\email{hohage@math.uni-goettingen.de})},\and
Milad Karimi\thanks{Corresponding author. Institute for Numerical and Applied Mathematics, Lotzestra{\ss}e 16-18, 37083 University of Göttingen, Germany  (\email{m.karimi@math.uni-goettingen.de})}, \and Björn Müller\thanks{Max-Planck Institute for Solar System Research, 37077 Göttingen, Germany (\email{muellerb@mps.mpg.de})}}

\usepackage{amsopn}
\newcommand{\mZ}{\mathbb{Z}}

\newcommand{\mK}{\mathbb{K}}

\newcommand{\ff}{\mathfrak{f}}

\newcommand{\cD}{\mathcal{D}}

\newcommand{\cF}{\mathcal{F}}

\newcommand{\cK}{\mathcal{K}}

\newcommand{\cO}{\mathcal{O}}

\DeclareRobustCommand{\rchi}{{\mathpalette\irchi\relax}}
\newcommand{\irchi}[2]{\raisebox{\depth}{$#1\chi$}} % inner command, used by \rchi

\newcommand{\Nfr}{N}

\DeclareMathOperator{\diag}{diag}
\DeclareMathOperator{\supp}{supp}
\DeclareMathOperator{\re}{Re}
\DeclareMathOperator{\im}{Im}
\DeclareMathOperator{\Cov}{Cov}
\DeclareMathOperator{\bCov}{\mathbf{Cov}}

\DeclareMathOperator{\R}{\mathbb{R}}
\DeclareMathOperator{\C}{\mathbb{C}}

\DeclareMathOperator{\OtK}{OtK}
\DeclareMathOperator{\KtO}{KtO}
\DeclareMathOperator{\HS}{\mathcal{HS}}

\DeclareMathOperator{\grad}{\mathbf{grad}}
\DeclareMathOperator*{\argmin}{argmin}

\DeclareMathOperator{\esssupp}{ess\,supp}

\DeclareMathOperator*{\chull}{ch}

\DeclareMathOperator{\Tr}{tr}
\DeclareMathOperator{\Diag}{Diag}

\DeclareMathOperator{\E}{\mathbb{E}}

\DeclareMathOperator{\Pois}{Pois}

\newcommand{\domX}{\mathbb{D}}
\newcommand{\domY}{\mathbb{M}}
\newcommand{\paren}[1]{\left(#1\right)}

\newcommand{\Nset}{\mathbb{N}}
\newcommand{\Rset}{\mathbb{R}}
\newcommand{\D}{\,\mathrm{d}} % d in integrals
\newcommand{\Cset}{\mathbb{C}}
\newcommand{\Var}{\textbf{Var}}
\newcommand{\norm}[1]{\left\| #1 \right\|}

\newcommand{\rank}{r}  %denotes the rank of a matrix
\newcommand{\cov}{\operatorname{cov}}

\newcommand\unnumberedfootnote[1]{{%
  \let\thempfn\relax% Remove footnote number printing mechanism
  \footnotetext{#1}% Add footnote text
}}

\newsiamthm{lemmadefinition}{Lemma and Definition}
\newcommand{\Zset}{\mathbb{Z}}
\makeatletter
\newcommand*{\addFileDependency}[1]{% argument=file name and extension
  \typeout{(#1)}% latexmk will find this if $recorder=0 (however, in that case, it will ignore #1 if it is a .aux or .pdf file etc and it exists! if it doesn't exist, it will appear in the list of dependents regardless)
  \@addtofilelist{#1}% if you want it to appear in \listfiles, not really necessary and latexmk doesn't use this
  \IfFileExists{#1}{}{\typeout{No file #1.}}% latexmk will find this message if #1 doesn't exist (yet)
}
\makeatother

\ifpdf
\hypersetup{
  pdftitle={Holographic X-ray Phase Contrast Imaging with Partial coherence: Uniqueness and reconstructions from 
  intensity correlations},
  pdfauthor={T. Hohage, M. Karimi, B. Müller }
}
\fi

\begin{document}
\maketitle
% REQUIRED
\begin{abstract}
Holographic coherent X-ray imaging enables nanoscale imaging of biological cells and tissues, rendering both phase and absorption contrast, i.e.\ real and imaginary parts of the refractive index. 
Unlike the standard model, which assumes a perfectly coherent incident beam, we consider partial coherence characterized by a known covariance operator. 
In addition, we assume time-resolved intensity measurements, granting access not only to expected intensities but also to their correlations. 
We investigate the information content of these correlations and analytically demonstrate that, under a symmetry-breaking condition on the sample and the illumination area, both phase and absorption contrast can be uniquely recovered in both the full and the linearized models.  
A key challenge in numerical reconstruction is the substantial increase in data dimensionality caused by computing intensity correlations during preprocessing. 
We propose a novel approach that leverages a low-rank assumption on the incident beam’s covariance operator, bypassing explicit correlation computation while still exploiting its full information. Numerical experiments demonstrate its feasibility, yielding accurate simultaneous reconstructions of phase and absorption contrast.
\end{abstract}

% REQUIRED

% REQUIRED

\section{Introduction}
Holographic X-ray phase contrast computed tomography has become a central tool in biomedical and material sciences~\cite{cloetens1999holotomography,paganin1998noninterferometric,wilkins1996phase}. 
It provides high-resolution 3D images of large volumes of quasi-transparent specimens such as 
biological tissues~\cite{Eckermann2021}.

In this paper, we consider the 2D imaging model
\begin{equation}\label{eq2}
    I_{f}=|\mathcal{D}(e^{f}u)|^{2},
\end{equation}
where $f:\Rset^2 \to \Cset$ is the unknown quantity of interest, 
$u:\Rset^2 \to \Cset$ describes the incident beam, 
$\mathcal{D}$ is a Fresnel-propagator (a unitary operator defined below), 
$|\cdot|^{2}$ has to be understood point-wise, and $I_{f}:\Rset^2\to [0,\infty)$ is the observed intensity. 
In contrast to the standard model, where $u$ is a 
deterministic function describing a spatially perfectly coherent incident beam $\tilde{u}(x_1,x_2,x_3)=u(x_1,x_2) e^{i\kappa x_3}$, we only assume partial coherence and model $u$ as a complex, circularly symmetric, centered  Gaussian process. 

The values $f(x_1,x_2)$ for any $x_1,x_2\in\Rset$ are given by 
line integrals over the perturbation of the refractive index of the sample in the propagation direction $x_3$ --- $\frac{1}{\kappa}\im f$ corresponding to the integrated real part of the refractive index (phase contrast) and $-\frac{1}{\kappa}\re f$ to the imaginary part (absorption). 
Under the projection approximation, valid for optically thin samples, $e^f u$ describes the electromagnetic field in a plane behind the sample. 
For $\mathcal{D}=I$ (intensity measurements directly behind the sample) and for a plane incident wave ($u\equiv 1$), all information on the phase contrast is lost, and only the absorption contrast can be recovered. To retrieve the 
phase of $e^f u$ and thereby phase contrast, the field 
is propagated to another parallel plane using the Fresnel propagator $\mathcal{D}$, derived from  the Helmholtz equation
under the Fresnel approximation. 
This approach, known as propagation-based or inline holographic phase contrast tomography, can be combined with tomographic techniques to recover the full 3D refractive index of the sample.

\subsection*{Motivation}
In practice, the coherence of an incident beam is never 
perfect but always partial, and the simplifying assumption of perfect coherence may lead 
to entirely erroneous or significantly suboptimal reconstruction results.  
An example where models based on perfect coherence 
completely fail is the \emph{self-amplified spontaneous emission} (SASE) process generating \emph{X-ray free electron laser} (XFEL) pulses \cite{hagemann2021single}. Such ultrashort pulses can even enable time-resolved 
imaging at extremely small time scales. 
Another motivation arises from laboratory sources, which exhibit significantly poorer coherence properties compared to synchrotron radiation, yet are indispensable for high-throughput experiments.

If the incident field $u$ is random, then obviously 
the measured intensities $I_{f}$ in \eqref{eq2} are random as well. 
The primary objective of this work is to investigate and utilize the information contained in intensity correlations 
\begin{align}\label{eq:defi_two_point_corr}
\operatorname{cov}[I_{f}](x,y) = \Cov(I_{f}(x),I_{f}(y)),
\end{align}
with a  particular focus on the additional information 
beyond the \emph{mean intensities} $\E[I_{f}(x)] =\sqrt{\operatorname{cov}[I_{f}](x,x)}$ (see eq.~\eqref{eq:E_Var}).  

Let us first consider the situation that 
$u=Z\widehat{u}$ with a complex-valued centered 
random variable $Z$ with finite second moments 
(Gaussianity is not needed here), corresponding to 
perfect spatial coherence. Then 
\[
\operatorname{cov}[I_{f}](x,y) = \operatorname{Var}(|Z|^2) \widehat{I}_{f}(x)\widehat{I}_{f}(y)
=\sqrt{\operatorname{cov}[I_{f}](x,x)\operatorname{cov}[I_{f}](y,y)}, 
\quad \widehat{I}_{f}:=|\mathcal{D}(e^f\widehat{u})|^2,
\]
so in this case, intensity correlations and mean intensities uniquely determine each other and hence contain the same information. 

To access the covariance function $\operatorname{cov}[I_{f}]$ in 
\eqref{eq:defi_two_point_corr}, we assume to 
have time resolved measurements of intensities 
$I_{f,n}=|\mathcal{D}(e^fu_n)|^2$ corresponding to a 
sequence of fields $u_1, u_2,\dots, u_N$ with the 
same distribution as $u$. Such measurements are 
possible with modern detectors exploiting the high sensitivity and rapid readout of two-dimensional photon-counting pixel arrays.
In other words, we 
partition the total photon dose into $\Nfr$ fractions. 
We then have to impose the weak ergodicity assumptions
\begin{align}
\begin{aligned}\label{eq:ergodicity}
    \E[I_{f}]&= \lim_{\Nfr\to\infty} \overline{I}_{f,\Nfr},&& 
\overline{I}_{f,\Nfr}:= \frac{1}{\Nfr }\sum_{n=1}^{\Nfr }I_{f,n},\\
      \operatorname{cov}[I_{f}]&= \lim_{\Nfr\to\infty}\widehat{\operatorname{cov}}_{\Nfr}[I_{f}],&&
      \widehat{\operatorname{cov}}_{\Nfr}[I_{f}]:=
\frac{1}{\Nfr }\sum_{n=1}^{\Nfr }(I_{f,n}-\overline{I}_{f,\Nfr})
(I_{f,n}-\overline{I}_{f,\Nfr})^\top.
  \end{aligned}
\end{align}

 Obviously, \eqref{eq:ergodicity} is satisfied if the samples 
$u_n$ are stochastically independent, which may  
reasonably be assumed for SASE pulses. 

For a coherent beam, i.e., deterministic $u$ in \eqref{eq2}, it is well known that the 
intensity $I_{f}$ cannot be observed directly 
but only a Poisson process with intensity $I_{f}$ describing photon counts. 
If $u$ is a Gaussian process, then 
conditioned on $u$ we also observe a Poisson process. 
The overall distribution of measured photon count 
data is then described by a Cox process, which will 
be considered in our numerical experiments, see \Cref{subsec:Discretization}. 
In such a model $\operatorname{cov}[I_{f}]$ can be recovered in  a joint limit 
where both $\Nfr$  and the number 
of photon counts per frame tend to $\infty$.

\subsection*{Contributions} 
On the theoretical side, we study the uniqueness of 
the inverse problem to reconstruct the projected refractive 
index $f$ from intensity correlations. We identify three 
unavoidable sources of non-uniqueness: the first is a 
global phase shift, and the second are local phase 
shifts by integer multiples of $2\pi$. The third is a 
bit less obvious and is caused by a twisted symmetry of 
$f$. To avoid this last type of non-uniqueness, we impose 
a condition which, roughly speaking, requires the support of 
$f$ to be contained in one half of the illuminated area. 
Then any two $f_1, f_2$ which lead to identical noise-free 
intensity correlation data must satisfy $e^{f_1-f_2}
=e^{ic}$ for some real constant $c$, i.e., $f$ is 
identifiable up to the first two sources of non-uniqueness. We also prove a uniqueness theorem for the linearized problem. 

The main challenge in the reconstruction process is the huge size of the correlations in high-resolution intensity maps. Note that $\operatorname{cov}[I_{f}]$ depends on four 
variables! Such correlations are typically too 
large to be stored and too expensive to be computed
in a pre-processing step. 
We propose a remedy for the case that the covariance operator of the incident beam $u$ has low rank 
which avoids any matrices in the image space of the forward problem. 
The complexity of our algorithm grows linearly 
in the number of pixels of the intensity correlation 
and quadratically in the rank. 

\subsection*{Related works} 
From the huge literature on phase retrieval problems, we only discuss 
some works related near-field holographic X-ray imaging as studied 
in this paper. For perfectly coherent sources (i.e., $u\equiv 1$ in \eqref{eq2}), the phase retrieval problem for general compactly supported samples admits a unique solution, provided at least two independent intensity patterns are measured at different object-to-detector distances \cite{jonas2003phase}. 

The uniqueness of phase contrast imaging with a single intensity measurement has been proven for single material objects if weak absorption and slowly varying phase shifts \cite{turner2004x} or small propagation distances, justifying the transport-to-intensity equation \cite{paganin2002simultaneous,teague1983deterministic} are assumed. 
Moreover, Nugent \textit{et al} \cite{nugent2007x} by referring to the \textit{phase-vortex} counterexample, discussed that two real-valued intensity measurements are not only sufficient but also necessary for unique reconstruction. 
However, Maretzke in \cite{maretzke2015uniqueness} disproved the widely believed existence of ambiguities in single detector near-field phase contrast imaging and  proved with the help of the theory of entire functions  that a compactly supported complex-valued function can be uniquely determined from intensity measurements only at \textit{single} detector distance and illumination wavelength. Stability estimates 
for the linearized problem were derived in 
\cite{MH:17}. 

Intensity correlation data were also studied in  \cite{bardsley2018imaging,bardsley2016kirchhoff} 
for a setup with unknown point scatterers 
where phases can be recovered in a preprocessing step. 
   
\subsection*{Outline} The rest of the paper is organized as follows: \Cref{Forward problem} is devoted to the imaging model and the derivation 
of an explicit formula for the forward operator, 
for which we then compute the Fr\'echet derivative 
and its adjoint in \Cref{Frechet derivative and its adjoint}. 
In \Cref{Linearized and Uniqueness} we discuss sources of non-uniqueness 
and prove the uniqueness result discussed above. Finally, our numerical algorithm dealing 
with  a realistic noise model based on a Cox process
is presented and tested on synthetic data in \Cref{Reconstructions}. 
Moreover, two short appendices are devoted to some supplementary issues.

\section{Forward problem}\label{Forward problem}
We will derive our theoretical results in arbitrary space dimensions $m\in \Nset$. 
The \emph{Fresnel transform} $\mathcal{D}:L^{2}(\R^{m})\longrightarrow L^{2}(\R^{m})$ can be defined by 
\begin{equation}
\label{Fresnel op}
    (\mathcal{D}f)(x):=\int_{\R^{m}}k_{\mathfrak f}(x-y)f(y)\,\mathrm{d}y,~~~\text{for all}~x\in\R^{m}
\end{equation}
(see, e.g., \cite{maretzke2018locality}). 
Here the convolution kernel $k_{\mathfrak f}:=c_m(\frac{\mathfrak f}{2\pi})^{m/2}n_{\mathfrak f}$ is given by the \emph{chirp function} 
\begin{equation}
    n_{\mathfrak f}(x):=e^{i\mathfrak f\frac{|x|^{2}}{2}}
\end{equation}
with the dimensionless \textit{Fresnel number} $\mathfrak{f}>0$ and the constant $c_m:=e^{\frac{-im\pi}{4}}$. 
There exists the following alternative expression for the Fresnel propagator, which will prove valuable in the following: 
\begin{equation}\label{Fresnel op-Alter}
   (\mathcal{D}f)(x)=c_m\mathfrak{f}^{\frac{m}{2}}n_{\mathfrak{f}}(x)\cdot\mathcal{F}(n_{\mathfrak{f}}\cdot f)(\mathfrak{f}x).
\end{equation}
Here $\mathcal{F}:L^{2}(\R^{m})\longrightarrow L^{2}(\R^{m})$ 
denotes the Fourier transform with the convention 
\begin{equation}\label{FT}
    (\mathcal{F}f)(\xi):=\frac{1}{(2\pi)^{m/2}}\int_{\R^{m}}f(x)e^{-i\xi\cdot x}\D x, ~~~~~\xi\in\R^{m}.
\end{equation}
With this definition the Fourier transform is unitary (see, e.g., \cite{reed1975ii}), which 
also implies that $\mathcal{D}$ is unitary. 

We assume that the object plane (in dimension $m=2$) can be restricted 
to an open, bounded set $\domX\subset \Rset^m$. This may be the ``illuminated 
area''
\begin{equation}\label{object domain}
    \domX:=\{x\in\mathbb{R}^m~:~\operatorname{cov}[u](x,x)>0\}.
\end{equation}
From \eqref{eq2} it is obvious that no information on $f$ is available 
in the non-illuminated region $\Rset^m \setminus \domX$ where $u=0$ almost surely. 
As mentioned in the introduction, $u$ is assumed to be a circularly symmetric Gaussian process with a known covariance $\cov[u]$. 
To mitigate the requirement that $\cov[u]$ is known on $\domX \times \domX$, a pinhole may be placed in the object plane. In such a setting $\domX$ 
describes the shape of the pinhole, which of course should be chosen large enough to contain the support of $f$.

Recall that circular symmetry of $u$ means that the distribution of $e^{i\alpha}u$ is independent of $\alpha\in\Rset$ and that 
the covariance operator $\bCov[u]$ of the random process $u$ can be defined implicitly by $\langle \bCov[u]\varphi_{1},\varphi_{2}\rangle:=\Cov(\langle u,\varphi_{1}\rangle,\langle u,\varphi_{2}\rangle)$ for all $\varphi_{1},\varphi_{2}\in L^2(\domX)$. 
%where we use the notation $\langle u, \varphi \rangle:=u(\varphi)$.
For a function $g\in L_{\Cset}^{\infty}(\domX)$ we define 
the  multiplication operator 
$M_{g}\in \mathcal{B}(L^2(\domX))$ by $M_{g}\varphi:=g \cdot \varphi$ for $\varphi\in L^2(\domX)$ and recall the definition of the Fresnel propagator 
$\mathcal{D}$ in \eqref{Fresnel op}. 
It is straightforward to see that $v_{f}:=\mathcal{D} M_{e^f}u$ is again a circularly symmetric Gaussian process with covariance operator 
\begin{equation}\label{eq: corr}
    \bCov[v_{f}]:=\bCov[\mathcal{D}M_{e^f}u]=\mathcal{D}M_{e^f} \bCov[u] M^{\ast}_{e^f} \mathcal{D}^{\ast}.
\end{equation}

Recall that a compact linear operator $\mathcal{K}$ on a Hilbert space $\mathbb{X}$ is called a 
\emph{Hilbert-Schmidt operator} if its singular values $\sigma_n(\mathcal{K})$, $n\in\Nset$ are 
square summable and that the set $\HS(\mathbb{X})$ of Hilbert-Schmidt operators on $\mathbb{X}$ equipped 
with the norm $\|\cK\|_{\HS}^2=\sum_{n=1}^{\infty}\sigma_n(K)^2$ is a Hilbert space. 
For the special case $\mathbb{X} = L^2(\domY)$ the \emph{kernel-to-operator} map defined by
\begin{align}\label{eq: HSIO}
\begin{aligned}
&\KtO:L^{2}(\domY\times\domY)\longrightarrow\HS(L^{2}(\domY)),\\
&(\KtO[k]\varphi)(x):=\int_{\domX}k(x,y)\varphi(y)\,\mathrm{d}y,~~ \text{for }~x\in\domY,~
\varphi\in L^2(\domY),~k\in L^{2}(\domY\times\domY)
\end{aligned}
\end{align}
is unitary (\cite[Thm. VI.18]{Reed1980}), so its inverse, the \emph{operator-to-kernel map} is given by 
\[
\OtK:=\KtO^{-1}=\KtO^{\ast}:\HS(L^2(\domY))\to L^2(\domY\times \domY). 
\]

\begin{assumption}
\label{ass:HilbertSchmidt}
    The covariance operator $\bCov[u]$ is a Hilbert-Schmidt operator on $L^{2}(\domX)$ with integral kernel $\operatorname{cov}[u] \in C(\domX\times\domX)$.
\end{assumption}

Using the identity $\Cov(|X|^2,|Y|^2)=|\Cov(X,Y)|^2$ 
for circularly symmetric complex Gaussian variables 
$X,Y$ (see \eqref{eq:Cov_circ_symm}) with $X:=v_{f}(x)$ and $Y:=v_{f}(y)$, we obtain the following relation between the two-point 
intensity correlations $\operatorname{cov}[I_{f}]$ in \eqref{eq:defi_two_point_corr} and the 
two-point correlations $\operatorname{cov}[v_{f}]$ of the phased fields $v_{f}$:
\[
\operatorname{cov}[I_{f}](x,y) = \big|\operatorname{cov}[v_{f}](x,y)\big|^2\quad 
\text{with}\quad\operatorname{cov}[v_{f}](x,y):=\Cov(v_{f}(x),v_{f}(y)).
\]
We also introduce an open, bounded subset $\domY\subset \Rset^m$ in the 
observation plane (again for $m=2$) in which intensity data are available.  
Using the previous identity and equations  \eqref{eq2}, \eqref{eq:defi_two_point_corr}, and \eqref{eq: corr}, the forward operator 
\begin{subequations}\label{par-to-ker map}
\begin{equation}
    F:L_{\Cset}^{\infty}(\domX)\longrightarrow L_{\Rset}^{1}(\domY\times\domY),~~~ F(f):=\operatorname{cov}[I_{f}]
\end{equation}
admits the following explicit representation:
\begin{equation}\label{eq:Nonlinear-Fourier-PRP}
F(f)= \big|\operatorname{cov}[v_{f}]\big|^2=\big|\OtK\bCov[v_{f}]\big|^{2}
= \big|\OtK(\mathcal{D}M_{e^f}{\bf Cov}[u]M^{\ast}_{e^f}\mathcal{D}^{\ast})\big|^{2}\,.
\end{equation}
\end{subequations}
We note that $\bCov[v_{f}]$ is a Hilbert-Schmidt operator (see \Cref{Adjoin-Frechet-diff}, part (i) below) such that the expressions in \eqref{par-to-ker map} are well-defined.   

We also note the following alternative expression for the forward operator
\begin{align}\label{eq:Fourier-PRP}
\begin{aligned}
F(f)(x,y) &=|\mathcal{F}_{2m}(K_f)(\mathfrak{f}x,-\mathfrak{f}y)|^2,\quad x,y\in\domY\\
&\text{with }K_f(p,q):= \mathfrak{f}^{m} n_{\mathfrak{f}}(p)e^{f(p)}K(p,q)e^{\overline{f(q)}}~\overline{n_{\mathfrak{f}}(q)}
\end{aligned}
\end{align}
where $\mathcal{F}_{2m}$ denotes the $2m$-dimensional Fourier transform. 
This follows from the identity $\operatorname{cov}[v_{f}](x,y)=n_{\mathfrak{f}}(x)(\mathcal{F}_{2m}K_{f})(\mathfrak{f}x,-\mathfrak{f}y)\overline{n_{\mathfrak{f}}(y)}$, $x,y\in\domY$ 
obtained from the form \eqref{Fresnel op-Alter} of the Fresnel propagator. 
This shows that our inverse problem can be seen as a Fourier phase retrieval 
problem in even dimensions with a rather special structure. 

The complete nonlinear inverse problem can then be formulated as 
stably reconstructing $f$ from sample correlations $\widehat{\operatorname{cov}}_{\Nfr}[I_{f}]$
defined in \eqref{eq:ergodicity} by approximately solving  the equation
\begin{equation}\label{eq:nonlinear_inverse_problem}
  F(f) \approx \widehat{\operatorname{cov}}_{\Nfr}[I_{f}].  
\end{equation}
%%%%%%%%%%%%%%%%%%%%%%%%%%%%%%%%%%%%%%%%%%%%%%%%%%%%%%
\section{Fr\'echet derivative and its adjoint}\label{Frechet derivative and its adjoint}
\subsection{Fr\'echet derivative of the forward operator}
The following lemma on well-definedness and Fr\'echet differentiability of the forward operator is  mostly straightforward and mainly formulated to introduce notation:

\begin{lemma}\label{lemma-Frechet diff}
 Let $f\in L_{\Cset}^{\infty}(\domX)$. Then the following hold true.
\begin{itemize}
\item[(i)] The operator 
\begin{equation*}
    \mathcal{C}(f):L_{\Cset}^{\infty}(\domX)\to\HS(L^{2}(\domY)), \quad f \mapsto  \bCov[v_{f}]
\end{equation*}
is well-defined by \eqref{eq: corr} 
and Fr\'echet differentiable with derivative 
\begin{equation}\label{eq:Frechet_CD}
\mathcal{C}'[f]h=\mathcal{D}M_{e^f}\mathcal{R}(h)M^{\ast}_{e^f}\mathcal{D}^{\ast},~~~~\text{for all}~~f,h\in L_{\Cset}^{\infty}(\domX).
\end{equation}
Here $ \mathcal{R}(h):=M_{h}{\bCov}[u]+{\bCov}[u] M_{\overline{h}} $, 
and all Banach spaces are considered as real Banach spaces. (Note that 
$\mathcal{R}$ and $\mathcal{C}'[f]$ are not complex linear!)
\item[(ii)]
The forward operator $F$ in \eqref{par-to-ker map} is well-defined and Fr\'{e}chet differentiable with 
\begin{equation}\label{eq:Frechet_F}
 F'[f]h=2\re\paren{\overline{\OtK\paren{ \mathcal{C}[f]}}\cdot \OtK\paren{ \mathcal{C}'[f]h}}~~\text{for } ~f,h\in L_{\Cset}^{\infty}(\domX),
\end{equation}
where $\re(f)(x):=\re(f(x))$ denotes the pointwise real part of a function $ f $.  
\end{itemize}    
\end{lemma}

\begin{proof}
\textbf{Part}(i). Note that for $ f\in L_{\Cset}^{\infty}(\domX) $ the operators $ M_{e^f} $ and $ M^{\ast}_{e^f}:L^{2}(\domX)\longrightarrow L^{2}(\domX) $ are bounded and  
$ \|M_{e^f}\|=\|M^{\ast}_{e^f}\|=\|e^f\|_{L_{\Cset}^{\infty}}<\infty $. It follows from \Cref{ass:HilbertSchmidt} and the fact that 
compositions of bounded and Hilbert-Schmidt operators are 
Hilbert-Schmidt that $\mathcal{C}[f]\in \HS(L^2(\domY))$. 
Consider the bilinear map $ \mathcal{B}:L_{\Cset}^{\infty}(\domX)\times L_{\Cset}^{\infty}(\domX)\longrightarrow \HS(L^{2}(\domX))$ defined by $(g,h)\mapsto\mathcal{B}(g,h):=M_g{\bCov}[u]M_h^{\ast}$ for all $g,h\in L_{\Cset}^{\infty}(\domX)$. 
Clearly, $ \mathcal{B} $ is well-defined 
and Fr\'echet differentiable with 
$\mathcal{B}'[g,h](\delta g,\delta h) = 
M_{\delta g}\bCov[u]M_h^{\ast} + M_g\bCov[u] M_{\delta h}^{\ast}$. Then the chain rule and the Fr\'{e}chet differentiability of the mapping 
$f\mapsto(e^f,e^f)$ yield \eqref{eq:Frechet_CD}.\\ 
\textbf{Part}(ii). The second part follows from the first part 
and the fact that the forward operator 
\[
 F=\mathcal{S} \circ \OtK \circ \mathcal{C},
\]
is a composition with the linear isomorphism $ \OtK $ and the pointwise squared modulus 
$\mathcal{S}:L_{\Rset}^{2}(\domY\times\domY)\longrightarrow L_{\Rset}^{1}(\domY\times\domY) $, $\mathcal{S}(f):=|f|^2$, which 
is Fr\'{e}chet differentiable (with $ L_{\Rset}^{2}(\domY\times\domY)$ considered as real Hilbert space) and 
$\mathcal{S}'[f]h = 2\re(\overline{f}h)$. 
\end{proof}

\subsection{Adjoint of Fr\'{e}chet derivative}
\label{sec: adjoint of frechet derivative}
For iterative regularization methods, we also need the 
adjoint of the Fr\'echet derivative. This is 
mostly standard, except for the adjoint of the 
mapping 
\begin{equation}\label{eq:operatorM}
\mathcal{M}:L_{\Cset}^{\infty}(\domX)\longrightarrow\mathcal{B}(L^{2}(\domX)),   \qquad  h\mapsto M_{h}
\end{equation}
which occurs in $\mathcal{C}'[f]$. 
In the discrete setting, it is evident that the adjoint of the mapping $\diag:\mathbb C^{m}\longrightarrow\mathbb C^{m\times m}$, 
defined by $h\mapsto\diag(h)$ with respect to Euclidean or Frobenius inner products, is the function 
$\Diag: \mathbb C^{m\times m}\to \mathbb C^m$ which 
maps a square matrix to its diagonal. 
Formulating a continuous analog is less obvious. 
Let $\mathcal{K}\in\HS(L^{2}(\domX))$ be a Hilbert-Schmidt integral operator corresponding to kernel $K\in L^{2}(\domX\times\domX)$. We are motivated to define $ (\Diag\mathcal{K})(x):=K(x,x) $ as the diagonal of the operator kernel $ K $. Since $ K\in L^{2}(\domX\times\domX) $ and the diagonal set $ \lbrace (x,x)\in\domX\times\domX~:~x\in\domX\rbrace $ have zero measure, the restriction of $ K $ to the diagonal is not well-defined. However, for the subspace 
$ \mathcal{S}_{1}(L^{2}(\domX))$ of operators $\mathcal{K}
\in \HS(L^2(\domX))$ for which the 
singular values of which are not only square summable but even summable, equipped with the norm 
$\|\mathcal{K}\|_{\mathcal{S}_1}:=\sum_{n=1}^{\infty}\sigma_n(\mathcal{K})$, 
two of the authors have recently shown in 
\cite{muller2024quantitative}
that there exists a unique bounded linear operator 
\[
\Diag:\mathcal{S}_{1}(L^{2}(\domX))\longrightarrow L^{1}(\domX) \quad  
\text{with}\quad 
\Diag(\mathcal{K})(x)=K(x,x)
\] 
for all $ x\in\domX $ and all operators $ \mathcal{K}\in\mathcal{S}_{1}(L^{2}(\domX))$ with continuous kernel $ K $. 
Operators $\mathcal{K}\in \mathcal{S}_1(L^2(D))$ are called 
\emph{trace class operators} since the  
trace of $\mathcal{K}$ is well-defined by 
$\Tr \mathcal{K}:=\sum_{n=1}^{\infty}(\mathcal{K}\varphi_n,\varphi_n)$ 
for any complete orthonormal system $\{\varphi_n:n\in\mathbb{N}\}$ in $L^2(\domX)$.
Moreover, 
$
\Tr(\mathcal{K})=\int_{\domX}\Diag(\mathcal{K})\,\mathrm{d}x.
$ 

The adjoint of $\mathcal{M}$ in \eqref{eq:operatorM} maps from 
$\mathcal{B}(L^{2}(\domX))'$ to $L_{\Cset}^{\infty}(\Omega)'$, and both 
of these spaces are inconvenient. However, 
$\mathcal{B}(L^{2}(\domX))$ has a nice predual: we have $\mathcal{S}_{1}(L^{2}(\domX)) ' = \mathcal{B}(L^{2}(\domX)) $
with respect to the dual pairing
$\langle A,B\rangle:=\Tr(B^{\ast}A)$ (see, e.g., \cite[Thm. VI.26]{reed2012methods}). 
It follows that $\mathcal{S}_{1}(L^{2}(\domX))\subset
\mathcal{S}_{1}(L^{2}(\domX))''=\mathcal{B}(L^{2}(\domX))'$ 
in the sense of the canonical embedding of a Banach space into 
its bidual. 
With these preparations, we can formulate the following proposition:
\begin{proposition}\label{diag-lemma}
Let $ \domX\subset\mathbb R^{m} $ be a bounded domain. Then the adjoint operator $ \Diag^{\ast}:L_{\Cset}^{\infty}(\domX)\longrightarrow \mathcal{B}(L^{2}(\domX)) $  is given by 
\begin{equation}
\Diag^{\ast}=\mathcal{M}.
\end{equation}
In particular, $ \mathcal{M}^{\ast}\Big|_{\mathcal{S}^{1}(\domX)}=\Diag$, and the restriction of $\mathcal{M}^{\ast}:\mathcal{B}(L^{2}(\domX))'\to L_{\Cset}^{\infty}(\Omega)'$ to 
$\mathcal{S}^{1}(\domX)\subset \mathcal{B}(L^{2}(\domX))'$ 
takes values in $L^1(\domX)\subset L_{\Cset}^{\infty}(\domX)'$.
\end{proposition}
\begin{proof}
Let $\mathcal{K}\in S_1(L^2(\domX))$ be an operator 
with integral kernel $K$ and let $h\in L_{\Cset}^{\infty}(\domX)$. Then the integral kernel of 
$M_h^{\ast}\mathcal{K}$ is given by $\overline{h(x)}K(x,y)$. Therefore,
\begin{align*}
\langle\Diag(\mathcal{K}),h\rangle_{L^{2}}
&=\int_{\domX}\Diag(\mathcal{K})(x)\overline{h(x)}\D x
= \int_{\domX} K(x,x)\overline{h(x)}\D x\\
&=\Tr(M_{h}^{\ast}\mathcal{K})
=\langle \mathcal{K},M_{h}\rangle
=\langle\mathcal{K},\mathcal{M}(h)\rangle.
\end{align*}
\end{proof}

We can now characterize the adjoint of the 
Fr\'echet derivative as follows:

\begin{proposition}\label{Adjoin-Frechet-diff}
Let $ f\in L_{\Cset}^{\infty}(\domX) $. Then the following holds true:
\begin{itemize}
\item[(i)] The adjoint $ \mathcal{C}'[f]^{\ast}:\HS(L^{2}(\domY))\longrightarrow L_{\Cset}^{\infty}(\domX)'$ takes values in the pre-dual space $ L^{1}(\domX)\subset L_{\Cset}^{\infty}(\domX)' $ of $ L_{\Cset}^{\infty}(\domX)$ and is given by
\begin{equation}
\mathcal{C}'[f]^{\ast} \mathcal{Q}=2\Diag\paren{\Re\paren{
M^{\ast}_{e^f}\mathcal{D}^{\ast}\mathcal{Q}\mathcal{D}M_{e^f}}\bCov[u]},
\end{equation} 
for all $\mathcal{Q}\in\HS(L^{2}(\domY))$ and $\Re(\mathcal{A}):=\frac{1}{2}(\mathcal{A}+\mathcal{A}^{\ast})$ for $\mathcal{A}\in \HS(\domX)$.
\item[(ii)] The adjoint $  F'[f]^{\ast}: L_{\Rset}^{\infty}(\domY\times\domY)\longrightarrow L_{\Cset}^{\infty}(\domX)' $ takes values in the pre-dual space $ L^{1}(\domX)\subset L_{\Cset}^{\infty}(\domX)' $ and  is given by 
\begin{equation}
 F'[f]^{\ast}A=2\mathcal{C}'[f]^{\ast} \KtO\paren{\OtK\paren{\mathcal{C}(f)} \cdot A },~~\text{for all}~~A\in L_{\Rset}^{\infty}(\domY\times\domY).\notag
\end{equation}
\end{itemize}
  \end{proposition}
\begin{proof}
\textbf{Part}(i). We write $ \mathcal{R}(h)=2\Re(M_{h}{\bCov}[u]) $ and note that $\Re$ is self-adjoint  in $\HS(L^2(\domX))$ and that the adjoint of 
$\HS(\mathbb{X})\to\HS(\mathbb{Y})$, 
$M\mapsto AMB^{\ast}$ with $A,B\in \mathcal{B}(\mathbb{X},\mathbb{Y})$ is given by $N\mapsto A^{\ast}NB$. 
Invoking \Cref{diag-lemma}, we obtain the assertion.\\
\textbf{Part}(ii). The Fr\'{e}chet derivative $  F'[f] $ is of the form $  F'[f]=\mathcal{S}'[\dots]\circ \OtK\circ \mathcal{C}'[f] $. This implies $  F'[f]^{\ast}=\mathcal{C}'[f]^{\ast}\circ \KtO\circ\mathcal{S}'[\dots]^{\ast}   $ and the assertion follows by substituting the partial operators.
 \end{proof}
\section{Uniqueness}\label{Linearized and Uniqueness}
\subsection{Unavoidable ambiguities}\label{unavoidable ambiguity}
A crucial aspect of the mathematical analysis of phase retrieval problems is 
identifying all potential ambiguities. 
In this discussion, we will identify unavoidable ambiguities inherent in both 
linear and non-linear inverse problems, and 
show that under certain conditions these are 
all sources of non-uniqueness.

\begin{lemma}[Non-uniqueness caused by global phase shifts]\label{lem:ambiguity}
Suppose that $K:=\operatorname{cov}[u]$ 
satisfies \Cref{ass:HilbertSchmidt}, and $K=\sum_{j=1}^JK_j$ 
with  $j=1,2,\dots,J,$ with $1\leq J\leq\infty$  
is the sum 
of positive semi-definite kernels 
$K_j$ with pairwise disjoint supports 
$\supp K_j\subset \mathcal{X}_j\times \mathcal{X}_j$.  
Then for any function $f\in L_{\Cset}^{\infty}(\domX)$ and any 
$h$ of the form 
\begin{align}\label{eq:kernel_fct}
h\in\mathcal{N}_{\{K_j\}}:=
\left\{\sum\nolimits_{j=1}^J i c_j \mathds{1}_{\mathcal{X}_j}: c_j\in\mathbb{R}\right\}
\end{align}
with indicator functions 
$\mathds{1}_{\mathcal{X}_j}(x)=1$ if $x\in\mathcal{X}_j$, $\mathds{1}_{\mathcal{X}_j}(x)=0$ else, we have 
\begin{align}\label{eq:ambiguity}
 F(f) = F(f+h)   \quad \text{and}\quad 
 F'[f]h = 0.
\end{align}
\end{lemma}

\begin{proof}
First, assume that $J=1$. Then 
$h$ is a constant function, $M_{e^h}$ reduces 
to a scalar multiplication, and hence, it commutes with 
$\bCov[u]$. Moreover, since $h$ is purely imaginary, $\overline{e^h} = e^{-h}$.
Therefore, 
\begin{align*}
M_{e^{f+h}}\bCov[u]M_{e^{f+h}}^{\ast}
&=M_{e^f} M_{e^h}\bCov[u]M_{\overline{e^h}}
M_{\overline{e^f}}\\
&= M_{e^f} \bCov[u]M_{e^h}M_{e^{-h}}
M_{\overline{e^f}}\\
&=M_{e^f}\bCov[u]M_{e^f}^{\ast}.
\end{align*}
This implies that $F(f+h)=F(f)$. It is easy to 
see that if $K=\sum_{j=1}^J K_j$  and $h$ is of the form \eqref{eq:kernel_fct}, then $\bCov[u]$ and $M_h$ still commute and $\overline{e^h} = e^{-h}$.
Hence, the equations above remain valid, and we can deduce that $F(f+h)=F(f)$. It follows that 
$F'[f]h = \displaystyle{\lim_{t\to 0}\frac{1}{t}(F(f+th)-F(f))=0}$.  
\end{proof}

We discuss two further sources of non-uniqueness:
\begin{remark}[Non-uniqueness caused by local phase shifts of multiples of $2\pi$]\label{rem:nonunique_local_shift}
Let $h\in L_{\Cset}^{\infty}(\domX)$ be a function with values in $2\pi i\Zset$ a.e. and
$f\in L_{\Cset}^{\infty}(\domX)$. Then $F(f+h)=F(f)$
since $e^{f+h}=e^f$ and $f$ only occurs in $F$ via $e^f$. 
\end{remark}
Note that we can avoid this source of non-uniqueness by confining ourselves 
to continuous functions $f$.

The non-uniqueness in \Cref{rem:nonunique_local_shift} does not occur 
in the linearized case, but the validity of the linearization becomes questionable 
for phase jumps in the order of $\geq 2\pi$. We discuss a third type of 
non-uniqueness for the nonlinear problem.
\begin{lemma}[Non-uniqueness caused by twisted symmetry]\label{lem:symmetry}
   Let $K$ satisfy the symmetry condition $K(p,q)=\overline{K(-p,-q)}$ and define $g(p):=\overline{f(-p)}-i\mathfrak{f}|p|^2$. Then $F(g)=F(f)$. 
\end{lemma}
\begin{proof}
Substituting
$e^{g(p)} = e^{\overline{f(-p)}} e^{-i\mathfrak{f}|p|^2}$ and 
$e^{\overline{g(q)}} = e^{f(-q)} e^{i\mathfrak{f}|q|^2}$ 
into $K_g$ and collecting phase terms gives
\begin{align*}
    K_g(p,q)&= \mathfrak{f}^{m} n_{\mathfrak{f}}(p)e^{g(p)}K(p,q)e^{\overline{g(q)}}~\overline{n_{\mathfrak{f}}(q)}\\
    &=\mathfrak{f}^m\overline{n_{\mathfrak{f}}(p)}e^{\overline{f(-p)}}K(p,q)e^{f(-q)}n_{\mathfrak{f}}(q)\\
    &=\mathfrak{f}^m\overline{n_{\mathfrak{f}}(p)}e^{\overline{f(-p)}}\overline{K(-p,-q)}e^{f(-q)}n_{\mathfrak{f}}(q)\\
    &=\overline{K_f(-p,-q)}~.
\end{align*}
Using the symmetry property of the Fourier transform $\overline{\cF_{2m}(\phi)} =\cF_{2m}(\overline{\phi(-\cdot)})$, we obtain 
 $\mathcal{F}_{2m}(K_g)
= \overline{\mathcal{F}_{2m}(K_f)}$. 
Taking squared moduli in \eqref{eq:Fourier-PRP} yields $F(g)=F(f)$.
\end{proof}
\subsection{Uniqueness of the nonlinear forward problem}
\label{sec: uniqueness of nonlinear forward problem}
\cref{lem:symmetry} shows that some assumption is needed which excludes symmetry. 
We will assume that the sample is contained 
``in one half'' of the illuminated area $\domX$ (see \eqref{object domain})
in the following sense (see Fig.~\ref{fig:uniqueness_ass}): 

\begin{assumption}\label{ass:one_side}
Suppose there exists a linear functional $\omega:\Rset^m\to \Rset$ and a constant $c\in\mathbb{R}$ 
 such that the ground truth $f\in L_{\Cset}^{\infty}(\domX)$ 
 satisfies
% \begin{subequations}
\begin{align}
\label{eq:ass_one_side}
&\inf \omega(\esssupp (f-ic))> \frac{1}{2}\left(\inf \omega(\domX) + \sup \omega(\domX)\right).
 \end{align}
\end{assumption}
Recall that the essential support is defined by 
$\esssupp f:=\domX\setminus\bigcup\{U:U\subset \domX$ relatively open,
$f|_U=0\text{ a.e.}\}$.  \Cref{ass:one_side} can be checked 
whenever a bound on the sample's support is known.  However, 
it has the drawback that, for a given beam, the size of the samples  
that can be imaged with uniqueness guarantees is reduced by a factor of $2$. 

\begin{assumption}\label{ass:covariance's supp}
    Assume that $\supp K = \domX\times \domX$ with $\domX$ defined in \eqref{object domain}.
\end{assumption}
Note that \Cref{ass:covariance's supp} excludes a nontrivial block decomposition of $K$ as in \Cref{lem:ambiguity}. Such a block structure is a highly unlikely 
scenario unless multiple independent beams are employed simultaneously; 
nevertheless, 
 \Cref{ass:covariance's supp} may still be violated for other reasons. A situation where \Cref{ass:covariance's supp} is 
highly plausible is the pinhole setting discussed in Section \ref{Forward problem} 
if the coherence length of the incident beam is larger than the diameter of the pinhole. 

We are now in a position to state our main theoretical result: 
\begin{theorem}[uniqueness]\label{theo:uniqueness}
Let \cref{ass:HilbertSchmidt} and \cref{ass:covariance's supp} be satisfied.
Then for $f_1,f_2\in L_{\Cset}^{\infty}(\domX)$ 
satisfying \Cref{ass:one_side}
we have 
\[
F(f_1)|_{\domY\times \domY} = F(f_2)|_{\domY\times \domY}\qquad \Rightarrow\qquad 
e^{f_1-f_2} = e^{i(c_1-c_2)}\quad \text{a.e.}
\]
for some real constants $c_1,c_2\in\Rset$.
\end{theorem}
Note that because of the non-uniqueness discussed in  \Cref{rem:nonunique_local_shift}, we cannot conclude that $f_1-f_2\equiv i(c_1-c_2)$ under the given assumptions. This would be possible, however, 
if we additionally assume continuity of $f_1$ and $f_2$.

As a first step, we establish the following lemma.
\begin{lemma}\label{lem:difference contrast and support prop} 
Let $f_{1},f_{2}\in L_{\Cset}^{\infty}(\domX)$. Then
    %\begin{itemize}
       %\item[\emph{(i)}]   
       \begin{align*}
       &\left(F(f_{1})-F(f_{2})\right)(x,y)=(2\pi)^{-\frac{m}{2}}\re\cF_{2m}(\overline{K_{\mathrm{sum}}(-\cdot)}*K_{\mathrm{diff}})(\mathfrak{f}x,-\mathfrak{f}y) \\
       \text{where } 
       &K_{\mathrm{sum}}:=K_{f_1}+K_{f_2},\qquad K_{\mathrm{diff}}:=K_{f_1}-K_{f_2},
       \end{align*}
       $K_{f_{j}}$, $j=1,2$ are introduced in \eqref{eq:Fourier-PRP}, and ``$*$" denotes the convolution operator.
\end{lemma}
\begin{proof} 
    Let $f_{1},f_{2}\in L_{\Cset}^{\infty}(\domX)$. 
    Then by \cref{eq:Fourier-PRP}, the Fourier convolution theorem, the identity $|z|^2-|w|^2=\re(\overline{(z+w)}(z-w))$, $z,w\in\Cset$ and $\overline{\cF_{2m}(K)} =\cF_{2m}(\overline{K(-\cdot)})$, we have
    \begin{align}\label{eq:real-part difference of FO}
    \begin{aligned}
        \left(F(f_{1})-F(f_{2})\right)(x,y)&=\re(\overline{\cF_{2m}(K_{\mathrm{sum}})}\cF_{2m}(K_{\mathrm{diff}})) (\mathfrak{f}x,-\mathfrak{f}y)\\
        &=(2\pi)^{-\frac{m}{2}}\re\cF_{2m}(\overline{K_{\mathrm{sum}}(-\cdot)}*K_{\mathrm{diff}})(\mathfrak{f}x,-\mathfrak{f}y).
    \end{aligned}
    \end{align}
    Note that we have extended $\cF_{2m}(K_{\mathrm{sum}})(\mathfrak{f}x,-\mathfrak{f}y)$ and $\cF_{2m}(K_{\mathrm{diff}})(\mathfrak{f}x,-\mathfrak{f}y)$ from $\domY\times\domY$ to all of $\Rset^{2m}$ by analytic continuation.
    \end{proof}
    
Let $\chull(A)$ denote the convex hull of a set $A \subset \R^m$, i.e., the smallest convex set containing $A$. 
The essential tool in the proof of the main uniqueness theorem is the following identity for the convex hull of the support of 
convolutions due to J.-L.~Lions
(\cite[Thm.~7]{lions1952supports}): if $k_{1}$ and $k_{2}$ are two distributions with compact support, then
\begin{align}\label{eq:chull_conv}
\chull \supp (k_1*k_2) = \chull \supp(k_1) \oplus \chull 
\supp(k_2),
\end{align}
where ``$\oplus$'' denotes \emph{Minkowski sum} and is defined by  $A\oplus B:=\{a+b:a\in A,b\in B\}$ for $A,B \subset \Rset^m$.

\begin{proof}[Proof of Theorem~\ref{theo:uniqueness}]
We may choose the coordinate system such that $\omega(x)=x_1$ and
\[
\alpha := \sup \omega(\domX) = - \inf \omega(\domX).
\]
Suppose that $F(f_1)|_{\domY\times\domY}=F(f_2)|_{\domY\times\domY}$. 
Set $h:=f_1-f_2$ and $c:=c_1-c_2$. We have to show that 
\begin{align}\label{eq:main_conclusion}
e^{h-ic}=1\quad \text{a.e.}
\end{align}
We assume the contrary that the essential support 
$\esssupp(e^{h-ic}-1)$ is not empty and proceed in several steps
to arrive at a contradiction. 
\paragraph*{Step 1:} The quantities
\begin{subequations}
\label{eq: definition beta, gamma}
\begin{align}
\beta_h&:= \inf\left\{p_1:p\in \esssupp (e^{h-ic}-1)\cap \domX\right\},\\
\gamma_h&:= \inf_{\tilde{c}\in\R}\sup\left\{p_1:p\in \esssupp (e^{h-i\tilde{c}}-1)\cap \domX \right\}
\end{align}
\end{subequations}
satisfy
\begin{subequations}
\begin{align}
\label{eq:beta_g0}
&\beta_h>0,\\
\label{eq:inf_gamma_attained}
&\exists \tilde{c}\in\R: 
\gamma_h=\sup\left\{p_1:p\in \esssupp (e^{h-i\tilde{c}}-1)\cap \domX \right\}\\
\label{eq:gamma_ge_beta}
&\gamma_h\ge \beta_h.
\end{align}
\end{subequations}
\Cref{eq:beta_g0} is a consequence of \Cref{ass:one_side}. 
\eqref{eq:inf_gamma_attained} holds true for any $\tilde{c}\in\R$ if 
$\gamma_h =\alpha$. Suppose that $\gamma_h<\alpha$. Then there 
exists $\hat{c}\in\R$ such that $\sup\left\{p_1:p\in \esssupp (e^{h-i\hat{c}}-1)\cap \domX \right\}<\frac{1}{2}(\gamma_h+\alpha)$. 
Note that the sets 
\begin{align}\label{eq:defi_domXt}
   \domX_t:=\{p\in\domX:p_1>t\},\qquad t\in (-\alpha,\alpha) 
\end{align}
are open since $\domX$ is open, and they have positive measure by the definition of $\alpha$. 
As $e^{h-i\hat{c}}=1$ a.e. on $\domX_{(\gamma_h+\alpha)/2}$, 
$\hat{c}$ is uniquely determined 
up to integer multiples of $2\pi$, and $e^{i\tilde{c}}=e^{i\hat{c}}$.
\eqref{eq:gamma_ge_beta} follows from our assumption that \eqref{eq:main_conclusion} is wrong if $e^{i\tilde{c}}\neq e^{i c}$. 
Otherwise, it is obvious.

\begin{figure}[htbp]
    \centering
     \begin{subfigure}{0.48\linewidth}
        \centering
        \includegraphics[width=\linewidth]{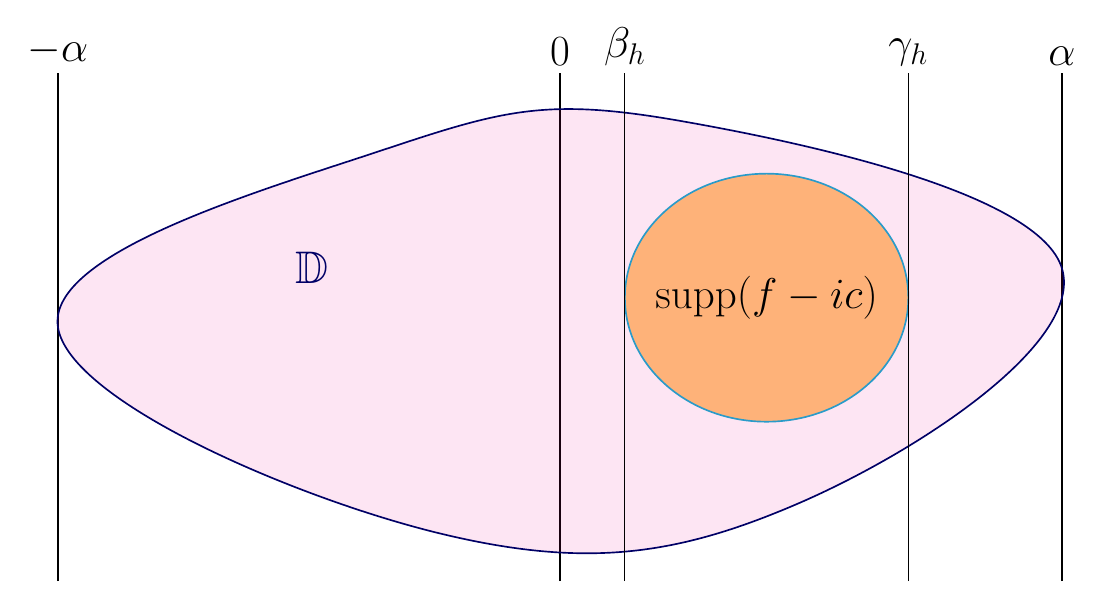}
        \caption{\label{fig:uniqueness_ass} Assumption \ref{ass:one_side} in the setting of the proof. After replacing $f$ with $h$, the image shows the definition of $\beta_h$ and $\gamma_h$ according to Eq.~\ref{eq: definition beta, gamma}.}
    \end{subfigure}
    \begin{subfigure}{0.50\linewidth}
        \centering
        \includegraphics[width=\linewidth]{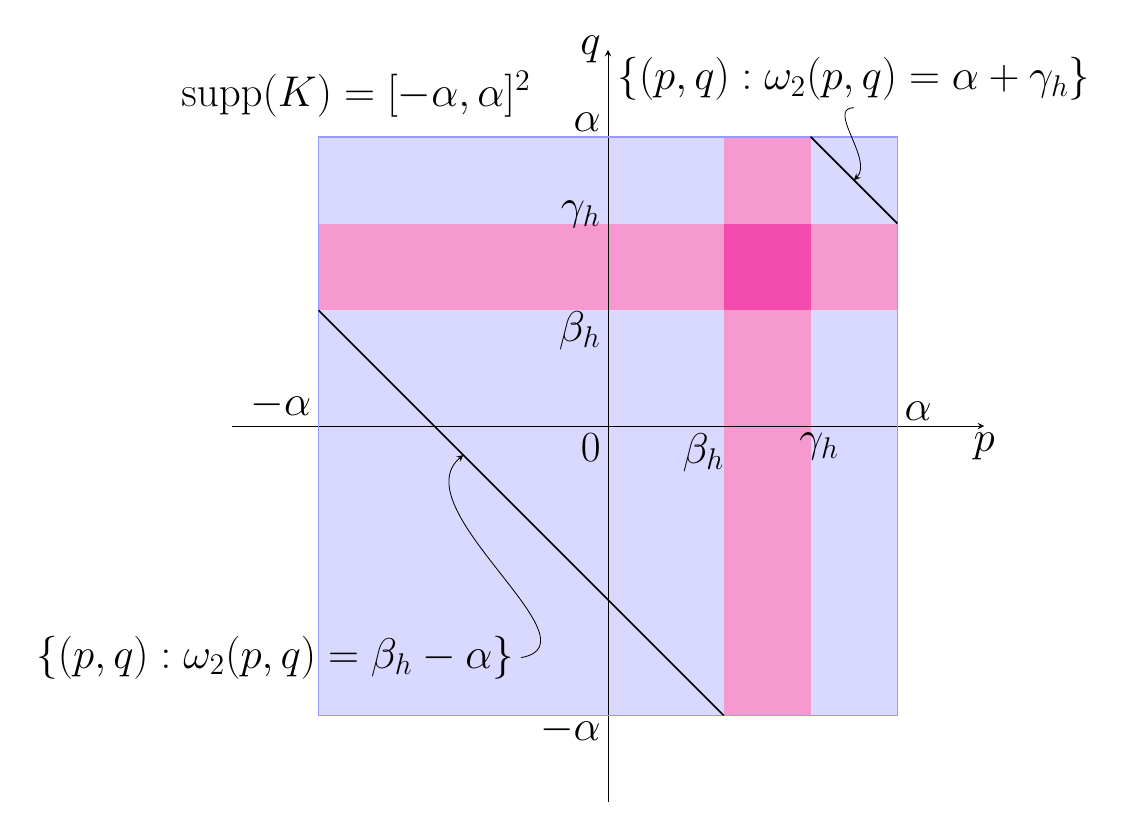}
        \caption{Steps 2 -- 4 of the proof sketched for $m=1$.}
    \end{subfigure}
    \hfill
    \caption{Geometric setting and constructions in the proof of Theorem \ref{theo:uniqueness}.}
    \label{fig:uniqueness-proof}
\end{figure}

\paragraph*{Step 2:}
Note that $K_{\text{diff}}(p,q) =  \mathfrak{f}^{m} n_{\mathfrak{f}}(p)K(p,q)~\overline{n_{\mathfrak{f}}(q)} \left(e^{f_1(p)+\overline{f_1(q)}}-e^{f_2(p)+\overline{f_2(q)}}\right)$
with the notation in \Cref{lem:difference contrast and support prop}.
To bound $\esssupp(K_{\text{diff}})$,  note that $\esssupp K_{\text{diff}}\subset \domX\times\domX$ and that for all $p,q\in\domX$, using \Cref{ass:covariance's supp}, we have the equivalences
\begin{align}\label{eq:Kdiff_equiv}
\begin{aligned}
K_{\text{diff}}(p,q)= 0&\Leftrightarrow
e^{f_1(p)+\overline{f_1(q)}}-e^{f_2(p)+\overline{f_2(q)}}= 0 \\
&\Leftrightarrow
f_1(p)+\overline{f_1(q)} - f_2(p) - \overline{f_2(q)} = h(p)+ \overline{h(q)}\in 
2\pi i \mZ\\
&\Leftrightarrow
\re h(p)= -\re h(q) \text{ and }
\im h(p)-\im h(q)\in 2\pi \mZ.
\end{aligned}
\end{align}

\paragraph*{Step 3:} We show that 
    \begin{align*}
\inf\omega_2(\esssupp K_{\text{diff}})\geq -\alpha+\beta_h
\quad \text{with }
        \omega_2(p,q)&:=\omega(p)+\omega(q)=p_1+q_1,~ p,q\in \mathbb{R}^m.
    \end{align*} 
Equivalently, we show that for almost all $(p,q)\in \R^{2m}$ the implication
\begin{align}\label{eq:Kdiff_lower}
\omega_2(p,q) < -\alpha + \beta_h \quad \Rightarrow\quad
K_{\text{diff}}(p,q) =0
\end{align}
holds true. We only need to consider $p,q$ with $p_1,q_1\in [-\alpha,\alpha]$. 
Then $p_1+q_1=\omega_2(p,q) < -\alpha + \beta_h$ implies $p_1\leq \beta_h$ 
and $q_1\leq \beta_h$. By the definition of $\beta_h$ we have 
$\re h(p)=0$ and $\im h(p)-c\in 2\pi\mZ$ and similarly for $q$ except for a nullset. By \eqref{eq:Kdiff_equiv} this implies $K_{\text{diff}}(p,q)=0$.

\emph{Step 4:} We show that
\begin{align}\label{eq:upper nonlinear}
\sup \omega_2(\esssupp(K_{\text{diff}})) \ge \gamma_h + \alpha.
\end{align}
We first consider the case $\gamma_h<\alpha$. Choose $\varepsilon>0$ such 
that $\gamma_h<\alpha-\varepsilon$. As $\gamma_h<\alpha-\varepsilon$, we have $e^{h(q)-i\tilde{c}}=1$ for a.a.\ $q\in\domX_{\alpha-\varepsilon}$ with $\tilde{c}$ from \eqref{eq:inf_gamma_attained}.  
By the definition of $\gamma_h$, the set $\cO:=\{p\in \domX_{\gamma_h-\varepsilon}\setminus\domX_{\gamma_h}:e^{h(p)-i\tilde{c}}\neq 1\}$ has 
positive measure. We have 
$\re h(p)+\re h(q)=\re h(p)\neq 0$ or $\im h(p)-\im h(q) = \im h(p)-\tilde{c} \notin 2\pi\mZ$ for all $p\in\cO$ and a.a.\ $q\in\domX_{\alpha-\varepsilon}$. 
By \eqref{eq:Kdiff_equiv} this shows that $K_{\text{diff}}(p,q)\neq 0$ for 
a.a.\ $(p,q)$ in the set $\cO\times \domX_{\alpha-\varepsilon}$, which has positive measure. This entails that $\omega_2(\esssupp(K_{\text{diff}}))>(\gamma_h-\varepsilon)+(\alpha-\varepsilon)$. As $\varepsilon>0$ can be arbitrarily small, 
this proves the claim for $\gamma_h<\alpha$. Now consider the case $\gamma_h=\alpha$ and choose $\varepsilon>0$.  
If $\cO_{\text{R}}:=\{p\in\domX_{\alpha-\varepsilon}:\re h(p)\neq 0\}$ has positive measure, then 
$(p,p)\in \esssupp K_{\text{diff}}$ for all $p\in\cO_{\text{R}}$ by 
\eqref{eq:Kdiff_equiv}. Hence, $\omega_2(p,p)\geq 2\alpha-2\varepsilon$, 
proving \eqref{eq:upper nonlinear} as $\varepsilon>0$ is arbitrary. 
Now assume that $\cO_{\text{R}}$ is a nullset. If there exists 
$(p,q)\in (\esssupp K_{\text{diff}}) \cap (\domX_{\alpha-\varepsilon}\times \domX_{\alpha-\varepsilon})$ with $K_{\text{diff}}(p,q)\neq 0$,  
then we again obtain $\omega_2(p,p)\geq 2\alpha-2\varepsilon$. 
Otherwise, we have $e^{i\im h(p)}=e^{i\tilde{c}}$ for 
almost all $p\in \domX_{\alpha-\varepsilon}$ and some $\tilde{c}\in\R$, which  
leads to the contradiction $\gamma_h\leq \alpha-\varepsilon$.

\paragraph*{Step 5:} 
Combining the last two steps and the fact that the convex hull of a set 
is the intersection of all half space containing it, yields 
\[
\omega_2(\chull \esssupp(K_{\text{diff}})) \subseteq [-\alpha+\beta_h, 2\alpha], \quad \text{and} \quad \gamma_h+\alpha \in \omega_2(\chull \esssupp(K_{\text{diff}})).
\]

\paragraph*{Step 6:}
Define $\Phi_{f_{1},f_{2}}:=\overline{K_{\mathrm{sum}}(-\cdot)}*
K_{\mathrm{diff}}$.    
Since the 
    Fourier-Laplace transform of a compactly supported 
    function is an entire function and since $\esssupp K_{f} =\supp K=\domX\times \domX$ 
    is bounded, 
$F(f_{1})|_{\domY\times\domY}$ and $F(f_{2})|_{\domY\times\domY}$ can be analytically extended to all of $\Rset^{2m}$. Therefore,  
\Cref{lem:difference contrast and support prop} implies that $\re \mathcal{F}_{2m}\Phi_{f_1,f_2}=0$. Hence, 
$\Phi_{f_1,f_2}$ is an anti-symmetric Hermitian function, i.e., $\overline{\Phi_{f_1,f_2}(-\cdot)}=-\Phi_{f_1,f_2}(\cdot)$.
Note that 
\[
K_{\text{sum}}(p,p)=\mathfrak{f}^mK(p,p)(e^{2\re f_1(p)}+e^{2\re f_2(p)})
\neq 0\qquad \text{for }p\in \domX,
\]
so $\omega_2(\esssupp (K_{\text{sum}}(-\cdot)) = [-2\alpha,2\alpha]$.  
Now, invoking the identity \eqref{eq:chull_conv}, noting that the support of an $L_{\Cset}^{\infty}$ function as a distribution 
corresponds to its essential support,  we find that 
\begin{subequations}
\begin{align}\label{eq:aux_ch_intv_a}
&\begin{aligned}
\omega_2(\chull \esssupp \Phi_{f_1,f_2}) &= \omega_2(\chull \esssupp (K_{\text{sum}}(-\cdot)) + \omega_2(\chull \esssupp(K_{\text{diff}}))\\
&\subset [-3\alpha + \beta_h, 4\alpha]\quad \mbox{and}
\end{aligned}\\
\label{eq:aux_ch_intv_b}
 &3\alpha + \gamma_h \in \omega_2(\chull \esssupp \Phi_{f_1,f_2}).
\end{align}
\end{subequations}
\paragraph*{Step 7:} 
As $\overline{\Phi_{f_1,f_2}}$ is anti-symmetric, 
$\esssupp \Phi_{f_1,f_2}$ and $\omega_2(\chull \esssupp \Phi_{f_1,f_2})$ 
must be point symmetric  with respect to the origin. 
In view of \eqref{eq:aux_ch_intv_b} this implies 
$-3\alpha-\gamma_h\in \omega_2(\chull \esssupp \Phi_{f_1,f_2})$. 
Using \eqref{eq:beta_g0} and \eqref{eq:gamma_ge_beta} we arrive at a 
contradiction to \eqref{eq:aux_ch_intv_a}. 
\end{proof}

We also state a corresponding uniqueness result for the linearized problem:
\begin{theorem}
    Let \cref{ass:HilbertSchmidt} and \cref{ass:covariance's supp} be satisfied and let $f,h\in L_{\Cset}^{\infty}(\domX)$ with 
    $h$ satisfying \Cref{ass:one_side}. Then 
\[
F'[f]h|_{\domY\times \domY} = 0\qquad \Rightarrow\qquad 
h= ic\quad \text{a.e.}
\]
for some real constant $c\in\Rset$.
\end{theorem}

\begin{proof} We only discuss the main differences to the proof of  
\Cref{theo:uniqueness}. We have 
\begin{align*}
(F'[f]h)(x,y) &= 2(2\pi)^{-m/2} \re \cF_{2m}\left(\overline{K_f(-\cdot)}*
K'_{f,h}\right)(\mathfrak{f}x,-\mathfrak{f}y)\\ 
&\mbox{with}\quad
K'_{f,h}(p,q):=\mathfrak{f}^m n_{\mathfrak{f}}(p)K(p,q)\left(h(p) + \overline{h}(q)\right)
 n_{\mathfrak{f}}(q).
\end{align*}
Again, $F'[f]h$ has an analytic extension from $\domY\times \domY$ to 
$\Rset^{2m}$. Under \Cref{ass:covariance's supp} we have 
\[
K'_{f,h}(p,q)=0 \quad \Leftrightarrow\quad 
\re h(p)=-\re h(q) \land  \im h(p) = \im h(q)
\]
instead of \eqref{eq:Kdiff_equiv}, and we define 
$\beta_h:=\inf\{p_1:p\in\esssupp (h-ic)\cap \domX\}$ and 
$\gamma_h:=\inf_{\tilde{c}} \sup\{p_1:p\in\esssupp (h-i\tilde{c})\cap \domX\}$. 
The remainder of the proof proceeds along the lines of the proof of 
\Cref{theo:uniqueness} with $K'_{f,h}$ in the place of $K_{\text{diff}}$ 
and $2K_f$ in the place of $K_{\text{sum}}$.
\end{proof}

\section{Reconstructions}\label{Reconstructions}
In this section, our focus is on effective numerical regularization methods to reconstruct jointly the phase and the absorption contrast.  Our numerical implementation unfolds in two steps: (i) discretization and noise model, (ii) iterative regularization method. 
\subsection{Discretization and noise model}\label{subsec:Discretization}
We assume that the illuminated area $\domX$ is contained in the square $[-1,1]^2$  
%\todo[inline]{Check chosen size of $\domX$. Important for Fresnel number!}
discretize the complex-valued object $f\in\Cset^{M_f\times M_f}$ as an equidistant grid  dividing it into $M_f\times M_f$ pixels. 
Similarly, the measurement 
domain $\domY$ is divided into $M_I\times M_I$ equidistant 
pixels. We implement the Fresnel propagator $\cD$ in 
the convolution form \eqref{Fresnel op} by FFT  
using zero-padding of $\domX$ to reduce periodization artifacts 
and restrictions in Fourier space to adapt the distance and number of detector pixels. 

The discrete forward operator is then given by 
\begin{equation}\label{GI-discrete}
F:\Cset^{M_f\times M_f}\to\Rset^{M_I^2\times M_I^2},~~~F(f):=\big|
\cD\diag(e^{f})\bCov[u]\diag(e^{\overline{f}})\cD^{\ast}\big|^2,
\end{equation}
where $\diag(e^f)\in \Cset^{M_f^2\times M_f^2}$ is the diagonal matrix with diagonal $e^f$. 

To generate synthetic data, we draw $\Nfr$ samples $u_{n}$, $n=1,\dots,\Nfr$ of the Gaussian random field $u$ and compute the corresponding intensities
\[
    I_{f,n}=|\cD \diag(e^{f}) u_{n}|^{2},~~f\in\Cset^{M_f\times M_f},~~n=1,\dots,\Nfr.
\] 
Observed intensities are affected by photon \emph{shot noise} described by a \emph{normalized Poisson process}  modeled as
\begin{equation}\label{eq:primary_data}
    I_{f,n}^{\text{obs}}\sim\frac{1}{t}\Pois(tI_{f,n}),\qquad n=1,\dots,\Nfr,
\end{equation}
where the parameter $ t > 0 $ can be
interpreted as the \emph{observation time} (proportional to the expected \emph{number of photon counts}) per image. The scaling factor $ \frac{1}{t} $ ensures that $\E[I_{f,n}^{\textrm{obs}}|u_n]=I_{f,n}$  (see \cite{hohage2016inverse}).

The $N$ intensity images of size $M_I\times M_I$ described in \eqref{eq:primary_data} are samples of a discrete Cox process and model our primary data. From these 
data we could in principle compute an estimator of the 
noise-free intensity correlations $\cov[I_f]$, which served 
as input data of the inverse problem in the previous sections by 
\[
    \widehat{\operatorname{cov}}_{N,t}[I_{f}]:=\frac{1}{\Nfr}\sum_{n=1}^{\Nfr}(I_{f,n}^{\text{obs}}-\overline{I}_{f,\Nfr}^{\text{obs}})(I_{f,n}^{\text{obs}}-\overline{I}_{f,\Nfr}^{\text{obs}})^{\top},~~\text{with}~~\overline{I}_{f,\Nfr}^{\text{obs}}:=\frac{1}{\Nfr}\sum_{n=1}^{\Nfr}I_{f,n}^{\text{obs}}.
\]
Correction terms on the diagonal due to the Poisson process  are discussed
in \cref{app:Cox-process}, but they are negligible in our setting since they scale like $\frac{1}{t}$ for large count rates. 

Since the size of  $\widehat{\operatorname{cov}}_{N,t}[I_{f}]$ is proportional 
to $M_I^4$, the computation of $\widehat{\operatorname{cov}}_{N,t}[I_{f}]$ 
severely limits the size of $M_I$ on available computational resources 
and impairs the practicality of the method. Therefore, it will be 
crucial to avoid an explicit computation of $\widehat{\operatorname{cov}}_{N,t}[I_{f}]$ and only work with the primary data $I_{f,n}^{\text{obs}}$.

\subsection{Iterative regularization method}\label{Avoid covariance fitting}

We wish to solve the discrete nonlinear inverse problem 
\begin{equation}\label{eq:discrete Inv-Probl}
     F(f)\approx \widehat{\operatorname{cov}}_{N,t}[I_{f}].
\end{equation}
To improve stability of reconstructions it is typically advantageous to 
incorporate any available prior information on the solution into the reconstruction process. 
In the following we will only consider the constraint that phase contrast for x-rays in non-positive and absorption contrast is non-negative~\cite{maretzke2019inverse} 
%\todo{add ref.}
, i.e., 
$f$ belongs to the set 
\begin{equation}\label{non-negativity cons}
    \mathbb{K}_{+}:=\{f\in\Cset^{M_f\times M_f}: -\re f\geq 0~\text{and}~\im f\geq 0\ \text{pointwise} \}.
\end{equation}
If additional support constraints are available, they are also easy to implement in this context. 
To solve \cref{eq:discrete Inv-Probl} both constraints are incorporated to the following Tikhonov regularization problem: 
\begin{equation}\label{eq:Tikh-minimization prob}
    \Bar{f}\in\argmin \left[\mathcal{S}(F(f))+\alpha \left(\rchi_{\mK_{+}}(f)+\norm{f-f_{0}}_{2}^{2}\right)\right],\qquad
    \mathcal{S}(g):=\frac{1}{2}\norm{g-\widehat{\operatorname{cov}}_{N,t}[I_{f}]}_{\HS}^{2}
\end{equation}
Here the norm $\norm{\cdot}_{\HS}$ in the \emph{data fidelity term} 
 denotes the Hilbert-Schmidt (or Frobenius) norm, $\alpha>0$ signifies the regularization parameter,  
 $f_{0}$ in the penalty term is an initial guess,
 and $\rchi_{\mK_{+}}$ denotes the \emph{characteristic function} of the set $\mK_{+}$, i.e., $\mK_+(f):=0$ if $f\in \mK_+$ and 
$\mK_+(f)=\infty$ else. 

\subsubsection*{Avoiding covariance fitting}
If we wish to minimize the Tikhonov functional \eqref{eq:Tikh-minimization prob} by gradient methods, we need to compute, in particular, the 
gradient of the data fidelity term, which is given by
\[
\grad (\mathcal{S}\circ F)(f) =  F'[f]^{\ast} F(f) - F'[f]^{\ast}\widehat{\operatorname{cov}}_{N,t}[I_{f}].
\]
If we wish to use Gau{\ss}-Newton type methods, i.e., approximate 
$\mathcal{S}(F(f))$ by $\mathcal{S}(F(f_n)+F'[f_n](f-f_n))$ in an iterative 
process, we additionally 
need to evaluate terms of the form $F'[f_n]^{\ast}F'[f_n]h$. 

However, evaluating 
any of these terms in a straightforward manner, in particular 
pre-processing intensities $I_{f,n}^{\text{obs}}$ to compute $\widehat{\operatorname{cov}}_{N,t}[I_{f}]$ prohibitively increases data dimensionality. In X-ray holography imaging, data sets can comprise millions pixels, which would result in the order of $10^{12}$ independent two-point correlations! 
This is  way too large to be stored in fast memory on most machines.  

The crucial point to make this computable for large problem instances is to assume that $\bCov[u]$ has small 
rank $\rank\ll m$  
and to avoid any 
$m\times m$ matrices and in particular any elements 
of the image space of the forward map. We need at most 
$m\times\rank^2$ or $\rank^2\times m$ matrices. 

This can be achieved by the following 
lemma providing a factorization 
of the pointwise squared modulus 
of a product of low-rank matrices 
into a product of low-rank matrices: 
\begin{lemma}\label{lemm:squared_modulus_factorization}
Let $\rank,m\in\Nset$ and  consider the mapping
\[
\tau:\C^{m\times\rank}\longrightarrow\C^{m\times(\rank\times\rank)},\qquad
[\tau(B)]_{i,p,q}:=B_{ip}\overline{B}_{iq}
\]
Then, with $|\cdot|^2$ applied element-wise, we have
\begin{align}\label{eq:squared_modulus_factorization}
|BC^{\ast}|^2=\tau(B)\tau(C)^{\ast},~~~~\text{for all}~~B,C\in\C^{m\times\rank},
\end{align}
\end{lemma}
\begin{proof}
For all $i,j=1,\cdots,m$, we have
\begin{align*}
|BC^{\ast}|_{ij}^{2}&=(BC^{\ast})_{ij}\odot\overline{(BC^{\ast})_{ij}}\notag\\
&=\paren{\sum_{p=1}^rB_{ip}\overline{C_{jp}}}\cdot\paren{\sum_{q=1}^r\overline{B_{iq}}C_{jq}}\notag\\
&=\sum_{p=1}^r \sum_{q=1}^r (B_{ip}\overline{B_{iq}})(\overline{C_{jp}}C_{jq})\notag\\&=\sum_{p=1}^r \sum_{q=1}^r [\tau(B)]_{i,p,q}[\overline{\tau(C)}]_{j,p,q}\notag.
\qed
\end{align*}
\end{proof}
To see how Lemma \ref{lemm:squared_modulus_factorization}
allows to avoid $m\times m$ matrices, let us introduce 
the function
\[
\Theta: \C^{m\times (\rank\times\rank)} 
\to \C^{m\times m}, \qquad 
\Theta(E):=EE^{\ast},
\]
for the matrix product such that \eqref{eq:squared_modulus_factorization} with $B=C$ becomes
\[
|BB^{\ast}|^2 = \Theta(\tau(B)).
\]
\begin{lemma}\label{Theta-lemma}
The matrix product $\Theta$ introduced above 
 \begin{itemize}
     \item [(i)] is Fr\'echet-differentiable with derivative 
     \[
     \Theta'\left[E\right]
     \begin{pmatrix}\delta E\end{pmatrix}=
   E(\delta E)^{\ast}+(\delta E)E^{\ast},\qquad 
     E, \delta E\in\C^{m\times (\rank\times\rank)};
     \]
     \item[(ii)] the adjoint $\Theta'\left[E\right]^{\ast}:\C^{m\times m}\longrightarrow\C^{m\times (\rank\times\rank)}  $ is given by
     \[
     \Theta'\left[E\right]^{\ast}(G)
     =(G+G^{\ast})E,\qquad G\in\C^{m\times m};
     \]
     \item[(iii)] the ``forward-backward operators" have the form
     \begin{align}
     \label{eq:forward-backward}
     \begin{aligned}
     &\Theta'\left[E\right]^{\ast} \left(\Theta(E)\right)
     =  E(E^{\ast}E)+(E^{\ast}E)E\\
     &\Theta'\left[E\right]^{\ast}\left(\Theta'\left[E\right](\delta E)\right)
     = 2E \left((\delta E)^{\ast} E\right) + 2\delta E (E^{\ast}E).
    \end{aligned}
     \end{align}
 \end{itemize}
 \end{lemma}
  \begin{proof}
The first part follows from the expansion 
 $\Theta(E+\delta E) = 
 \Theta(E)+ E(\delta E)^{\ast}+(\delta E)E^{\ast} +
 (\delta E)(\delta E)^{\ast}$ and the fact that 
 $(\delta E)(\delta E)^{\ast} = \mathcal{O}(\|\delta E\|^2)$. 
 The other parts are straightforward consequences. 
 \end{proof}
 
The key point about the formulas in  \Cref{Theta-lemma} is that thanks to the bracketing 
in \eqref{eq:forward-backward} they completely avoid the formation  of 
$m\times m$ matrices.

If the covariance matrix $\bCov[u]$ is 
 of rank $\rank\ll m$, then it has a factorization 
 $\bCov[u] = VV^{\ast}$ where $V$ has $\rank$ columns and the forward operator corresponding to the correlation data reads as
 \[
F(f) := |B(f)B(f)^{\ast}|^{2} = \Theta(\tau(B(f)))\qquad \mbox{with} \qquad 
B(f) :=\mathcal{D} \diag(e^f)V, 
\]
and by the chain rule we can compute $F'[f]^{\ast}F'[f]$
using matrices no larger than $m\times \rank^2$. 

The forward operator corresponding to mean intensity data is given by 
\[
F_{\mathrm{mean}}(f) := \Diag|B(f)B(f)^{\ast}|^{2} = \Xi(B(f))
\qquad \mbox{with}\qquad 
\Xi(B):=\left(\sum\nolimits_{j=1}^{\rank}|B_{kj}|^2\right)_k.
\]
\begin{figure}
    \centering
    \includegraphics[width=1\linewidth]{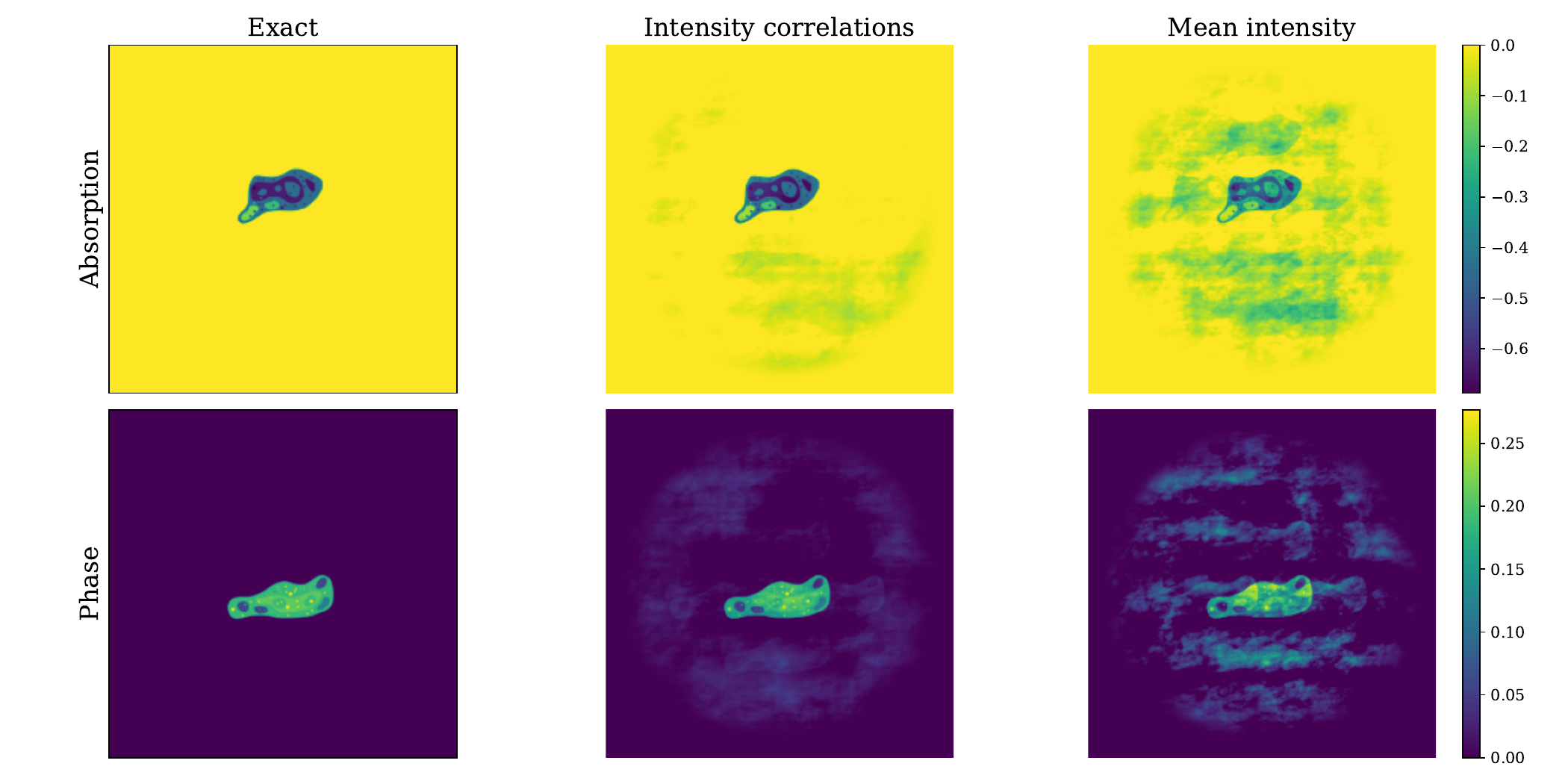}
    \caption{\small Comparison of holographic X-ray phase contrast imaging from intensity correlations and mean intensity for joint reconstructions of phase $\im f$ and absorption $-\re f$.}
    \label{fig:comparision_results}
\end{figure}
\begin{figure}
\centering
\includegraphics[width=\linewidth]{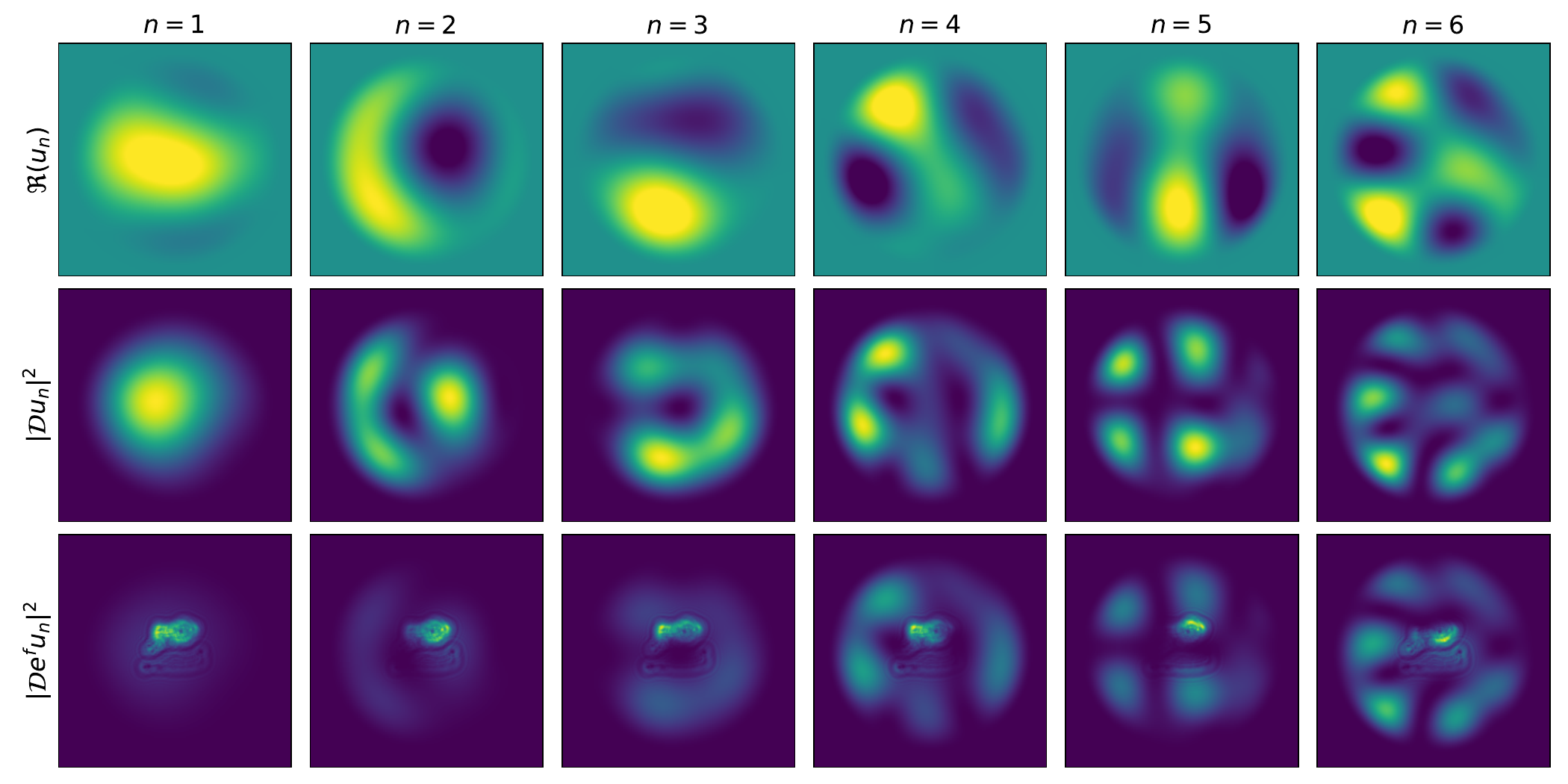}
{\caption{\label{fig:show_measurements}\small{ Visualization of the partial coherent incident beam and the fractionated data.}}}
\end{figure}
\subsection{Numerical implementation}
\subsubsection*{Software} 
All reconstructions were produced by the inverse problems python library ``RegPy"~\cite{regpy}. It provides tools to implement custom forward models as well as a variety of regularization methods and stopping rules. Further details can be found on Git \url{https://github.com/regpy/regpy}. 
\subsubsection*{Minimization of the Tikhonov functional}
The minimization problem in \cref{eq:Tikh-minimization prob} is solved by the generalized \emph{Fast iterative shrinkage-thresholding algorithm} (G-FISTA) \cite{nesterov1983method,beck2009fast}. To improve convergence and stability, we employ a homotopy (continuation) strategy in the regularization parameter $\alpha$, where the problem is solved sequentially for decreasing values of $\alpha$, each initialized with the previous iterate. We choose regularization parameters $\alpha_{k}=\alpha_{0} c^{k}$ for some $c\in (0,1)$ to decrease gradually with initial regularization parameter $\alpha_{0}$ for both intensity correlations and  mean intensity.

\subsubsection*{Reconstruction results}
\label{sec: reconstruction results}
We numerically analyze the performance of G-FISTA for holographic X-ray phase contrast imaging from both intensity correlations and mean intensity data.
To demonstrate the feasibility of this approach, we implement the forward operators corresponding to intensity correlations and mean intensity data for a two-dimensional cell test pattern with $256 \times 256$ pixels introduced in \cite{giewekemeyer2011x} as contrast $f$. 
To produce synthetic data reconstructions, we use a partially coherent beam and $N=3000$ frames. 
The regularization parameters are chosen according to $\alpha\in \{10^{-9}\cdot(\frac{1}{3})^k: k\in\Nset\}$. 
For mean intensity, the inversion with G-FISTA is started with an initial guess $f_{0}=0$, while for intensity correlations, a warm start from mean intensity reconstruction is chosen. 
For the holographic regime, a Fresnel number $\ff=\frac{10}{2\pi}$ is chosen for both forward operators.  

To construct $V \in \Cset ^{m\times r}$, we generate $r=4$ independent samples of smooth random fields as follows: we draw i.i.d. Gaussian Fourier coefficients, multiply them by a rapidly decaying spectral filter of the form $e^{-\frac{|\xi|^{2}}{\sigma^{2}}}$ with $\sigma=0.5$, and apply an inverse Fourier transform to obtain spatially smooth realizations.
A spatial cutoff is then applied to localize the fields.
The resulting samples are orthonormalized via a singular value decomposition to obtain the columns of $V$.

Throughout the regularization process, we apply the $L^{2}$- norm for both data fidelity and penalty term together with a non-negativity constraint for both phase and absorption contrasts. 
Additionally, the data are polluted with shot noise described by a Cox process with observation time $t=10^{9}$ photon counts per image.    

A comparison of the holographic X-ray phase contrast imaging from intensity correlations and mean intensity data is shown in \Cref{fig:comparision_results}. 
The joint reconstruction results for mean intensity are achieved using a 5-stage process with 100 G-FISTA iterations per stage, while for intensity correlations, we use a 2-stage process with 50 G-FISTA iterations per stage. 
The results show that intensity correlations enable simultaneous reconstruction of both phase and absorption contrasts, whereas mean intensities alone do not yield faithful reconstructions.
It appears that in the case of mean intensity data, some \emph{ghost-like} patterns produce artifacts in the reconstructions. 
The comparison demonstrates that a \emph{significant} amount of additional information can be achieved by the correlation data compared to the mean intensity data. 
Numerical simulations, complemented by preliminary theoretical results, suggest that the reconstructions become more accurate as the number of frames, $\Nfr $, increases. 
 
Additionally, we show in \Cref{fig:show_measurements} the 6 eigenvectors (or principal components) with the highest eigenvalues.
These components correspond to the dominant singular values and capture the principal variations of the input random fields, which is consistent with the stochastic nature of SASE pulses. Each component maximizes the variation of the random fields in a subspace, which is orthonormal to the previous components.
According to these results, there is a good agreement between theory and practice, as can be seen by \emph{all-at-once} phase and absorption reconstruction.    

\section{Conclusion and outlook}
 We have theoretically studied holographic X-ray imaging by exploring and exploiting the information content of the intensity correlations. Specifically, we proved unique identifiability for the nonlinear problem up to unavoidable 
 sources of non-uniqueness.

 For large-scale problems, the increased data dimensionality giving rise to a major difficulty when we compute the correlations from the pre-processing intensities. To remedy this we have avoided covariance fitting in the process of regularization by assuming  that the covariance matrix is \emph{known} and has a low rank. By \Cref{lemm:squared_modulus_factorization} and \Cref{Theta-lemma}, we showed that the quantitative holographic imaging can then effectively be interpreted as the application of forward-backward propagation, exploiting implicitly the full information content of correlation data. Moreover, the computational cost of the algorithm scales \emph{linearly} in the pixel number $m$ and \emph{quadratically} in the rank $\rank$ of $\bCov[u]$.

 Let us close this manuscript by introducing  possible directions for future research: One such direction is to replace the ``symmetry-breaking condition'' 
 \eqref{eq:ass_one_side} by different conditions that may be better suited for 
 specific applications. Furthermore, one could analyze the analogous, mathematically more 
 challenging problem for the full Helmholtz equation instead of the 
 Fresnel approximation, which might be necessary in different applications 
 with smaller wave numbers of the incident beam. 
 Another natural aim would be to quantify the stability of \cref{eq:nonlinear_inverse_problem} to find $f$ from observed correlations depending on the \textit{cross-covariance}, the \textit{Fresnel number}, and on the \textit{a-priori} information. A further direction is to extend our methodology to meet our forward problem when $\bCov[u]$ is \emph{unknown}. This is more applicable in real world problems. The idea is to first estimate the covariance matrix of the empty beam by modal decomposition of SASE pulses and then exploiting the information in an advanced method to get faithful reconstruction.    
\appendix 

\section{Circularly symmetric Gaussian random vectors}\label{complex-valued random variables}
Recall that a random variable $Z$ with values in $\Cset^m$ is called 
\emph{complex Gaussian} if $(\re Z,\im Z)$ is multivariate Gaussian 
in $\Rset^{2m}$ and \emph{circularly symmetric} if  
$e^{i\varphi}Z$ has the same distribution as $Z$ for all $\varphi\in \R$. 
Obviously, if the expectation of a circularly symmetric random vector 
exists, it must vanish. Let $Z$ be a circularly symmetric Gaussian 
random vector. Then $\E[ZZ^{\top}]=\E[e^{i\varphi}Ze^{i\varphi}Z^{\top}] 
= e^{2i\varphi} \E[ZZ^{\top}]$, so 
\begin{equation}\label{eq:aux_circ_symm}
\E[Z Z^{\top}]= 0\quad \mbox{if $Z$ is Gaussian and complex symmetric.}  
\end{equation}
In particular, if $Z$ is scalar, then taking the imaginary part of 
$\E[Z^2]=0$ shows that $\re Z$ and $\im Z$ are uncorrelated, 
which by Gaussianity implies that they are independent. 

If $X$ and $Y$ are circularly symmetric Gaussian random variables, then \emph{Isserlis' theorem} \cite{Isserlistheorem1918} yields the identity
\[
\E\left[X\overline{X}Y\overline{Y}\right]
= \E[X\overline{X}]\E[Y\overline{Y}] + \E[X\overline{Y}] \E[\overline{X}Y] + \E[XY]\E[\overline{X}\overline{Y}].
\]
The last term vanishes due to \eqref{eq:aux_circ_symm}, and the remaining terms can be rearranged to 
\begin{equation}\label{eq:Cov_circ_symm}
\Cov(|X|^2,|Y|^2)=|\Cov(X,Y)|^2.
\end{equation}
In particular, for $X=Y$ using $\E[X]=0$ we obtain
\begin{equation}\label{eq:E_Var}
    \Var(|X|^2) = \E[|X|^2]^2.
\end{equation}

\section{Cox processes}\label{app:Cox-process}
Given a random non-negative function (or measure) $I$ on a domain $\domY$, a 
Cox process $P$ with mean intensity $I$ is a point process $P=\sum_{i=1}^N \delta_{x_i}$ which, conditioned on $I$ is a Poisson process with intensity $I$. 
Here both the points $x_i\in\domY$ and the total number of points $N$ are random. 
For a more thorough characterization of Cox processes we refer to \cite{grandell:76}.
For a continuous function $f:\domY\to \Rset$, let 
\[
\langle P,f\rangle:=\sum\nolimits_{i=1}^N f(x_i).
\]
If $c_I(x,y):=\Cov(I(x),I(y))$ and $f,g:\domY\to \Rset$ 
are two continuous functions we have 
\[
\Cov\left(\langle P,f\rangle,\langle P,g\rangle\right)
= \int_{\domY}\int_\domY c_I(x,y)f(y)g(x)\,\mathrm{d}y\mathrm{d}x 
+ \int_{\domY} (\E[I])(x) f(x) g(x)\,\mathrm{d}x
\]
(see \cite{grandell:76} for the case of indicator functions $f,g$ which implies 
the above formula by density). In this sense we have 
\[
\Cov[P] = \Cov[I] + M_{\E[I]}.
\]
Note that $\Cov[I]$ is quadratic in $I$ whereas $\E[I]$ is linear in $I$. 
For the count rates considered in this paper we found the second term to 
be negligible compared to the first. 

Concerning uniqueness the correction term $M_{\E[I]}$ does not play a role 
since its Schwartz kernel is supported only on the diagonal $\{(x,x):x\in\domY\}$, 
which is a nullset in $\domY\times \domY$. 
\section*{Acknowledgments}
We would like to thank Tim Salditt for many helpful discussions. 
Financial support by Deutsche Forschungsgemeinschaft (DFG, German Research Foundation)
through Grant 432680300, SFB 1456— Mathematics of Experiments, project C03, is gratefully acknowledged.
\end{document}